\documentclass[12pt]{amsart}
\usepackage[english]{babel}
\usepackage{amsmath}
\newtheorem{defi}{Definition}[section]
\newtheorem{exa}{Example}[section]
\newtheorem{lema}{Lemma}[section]
\newtheorem{teo}{Theorem}[section]
\newtheorem{rem}{Remark}[section]
\newtheorem{coro}{Corollary}[section]
\newtheorem{pro}{Proposition}[section]

\newcommand{\su}{$SU_q(2)$}

\begin{document}

\title[GLOBAL PSEUDO-DIFFERENTIAL OPERATORS ON THE QUANTUM GROUP  $SU_q(2)$]
 {GLOBAL PSEUDO-DIFFERENTIAL OPERATORS ON THE QUANTUM GROUP  $SU_q(2)$}

\author[Carlos A. Rodriguez T.]{Carlos A. Rodriguez Torijano}

\address{Department of Mathematics\\
Universidad de los Andes\\
Bogot\'{a}\\
Colombia}

\email{ca.rodriguez14@uniandes.edu.co}




\date{\today}
\begin{abstract}
In this paper, following \cite{RSM}, we develop the theory of global  pseudo-differential operators defined on the quantum group $SU_q(2)$, and provide some spectral results concerning  these operators. We define a graduation for  this algebra of pseudo-differential operators  in terms of its natural Fourier decomposition and, using the infinite-dimensional representations introduced by Woronowicz \cite{W}, we also provide a $*$-representation of $SU_q(2)$ as bounded pseudo-differential operators acting on $L^2(\mathbb{S}^1)$.
\end{abstract}

\maketitle

\section*{\textbf{INTRODUCTION}}

The interest in quantum groups arises in a very first form during the early eighties. The  works of Drinfield and Jimbo \cite{Drinf} introduced formally the concept as a particular and important example of Hopf algebras. Just one year later, Woronowicz  (\cite{W} and \cite{W2}) presented a complete treatment of the quantum group $SU_q(2)$ by establishing the representations of this algebra and proving the existence of a Haar state associated to it. Being deformations of more rigid structures (as Lie groups and their Lie algebras), it is natural to think of the use of the  generalization of the functional analysis on these objects, associated to their representation theory via the Peter-Weil theorem, to associate to it a pseudo-differential calculus. \\�\\ In  \cite{RSM} the authors introduce a natural class of global pseudo-differential operators on compact quantum groups, in terms of  their distribution theory and the corresponding Fourier analysis, which follows the lines of the already accomplished theory for compact Lie groups in \cite{R.Tbook}, and generalized to homogeneous spaces and compact manifolds in \cite{DR}. On the one hand, in \cite{RSM} this construction is connected with the
algebraic differential calculi on Hopf subalgebras of compact quantum groups, the particular case of the quantum group $SU_q(2)$ being treated in detail and, on the other hand, it is also studied from the point of view of spectral noncommutative geometry \cite{Connes 2} through the notion of spectral triple, which has been also used in different contexts to understand quantum groups as ``noncommutative manifolds" (see e.g.  \cite{Goswami}, \cite{VS}, \cite{BK}, \cite{Connes}and \cite{ConnesSUq(2)}). 
\\ \\
In this paper we will use the ideas of the global pseudo-differential calculus developed by Ruzhansky and Turunen in \cite{R.Tbook} to introduce the concept of pseudo-differential operators defined on the quantum group $SU_q(2)$.  As in   \cite{RSM}, the main tool for this purpose is the version of the Peter-Weyl theorem for this algebra. In contrast with the treatment in   \cite{RSM}, we will define a Fourier order for these pseudo-differential operators in terms of its natural Fourier decomposition, inducing a graduation on this algebra, and we will also provide a formula for the composition in terms of homogeneous components of their symbols. Finally, we will give a representation of  $SU_q(2)$ on the algebra of pseudo-differential operators on $\mathbb{S}^1$  using the infinite-dimensional representations introduced by Woronowicz in \cite{W}.\\ 
\\
Let us describe the contents of the paper. In section 1 we recall some basic results on the  Hopf algebra structure used to define the quantum group $SU_q(2)$, and other aspects of the theory necessary to introduce the basic objects of the analysis (Fourier inversion formula and Plancherel's Identity) used in the definition of the global pseudo-differential calculus (we refer the reader to  \cite{RSM} for more on the Fourier analysis on compact quantum groups). In section 2 we focus on the $*$-representations of $SU_q(2)$  and we construct explicitly a representation on bounded pseudo-differential operators parametrized by the circle. We also prove a result on the distinction  between  the quantum groups $SU_q(2)$ depending on algebraic properties of the parameter $0<q<1$.
In section 3, the interest is focused in the study of the global pseudo-differential operators on $SU_q(2)$ and the graduation induced by the natural Fourier order for the symbols of these operators. A composition formula is established for the corresponding symbol classes. 
In section 4 we prove some results concerning spectral  properties of special type of global pseudo-differential operators, in particular  a result about the Fredholm  index of  global pseudo-differential operators by certain symbols.

 \section{Fourier Analysis on $SU_q(2)$}\label{S:FAonSUq(2)}
 
  In this section we treat the basics on the Fourier analysis for $SU_q(2)$. Recall that, for $0<q<1$, the  quantum group $SU_q(2)$ is the Hopf {*-}algebra generated by the following relations between generators $a,b,c, d$ (see  \cite{KS}, \cite{K} and references therein): $ab=qba,$  $bd=qdb,$ $cd=qdc,$  $bc=cb,$  $ad-qbc=da-q^{-1}bc=1$, with an involution given by $a^*=d,$ $b^*=-qc,$ $c^*=-q^{-1}b,$ and $d^*=a$, plus the Hopf algebra structure is specified as follows:
\begin{itemize} 
\item[i.] $\Delta(a)=a\otimes a+ b\otimes c,$ $\Delta(b)=a\otimes b+ b\otimes d,$ $\Delta(c)=c\otimes a+ d\otimes c,$ $\Delta(d)=c\otimes d+ d\otimes d,$

\item[ii.] The antipodal application is defined by $S(a)=a^{*},$ $S(b)=c,$ $S(c)=-qc^{*},$ and $S(d)=d^{*}$

\item[iii.] The co-unit acts by $\epsilon(a)=1=\epsilon(d)$ and $\epsilon(c)=0=\epsilon(b).$
\end{itemize}

The quantum group $SU_q(2)$ is often thought as a `twisted" version of (the algebra of functions on) the compact Lie group $SU(2)$. If we write
\[ u=\begin{bmatrix}
   a & b \\
    c & d
\end{bmatrix}
=\begin{bmatrix}
   a & -qc^{*} \\
    c & a^{*}
\end{bmatrix}\]
we can see that $SU_q(2)$ has two generators as a $C^{*}$- algebra, namely {\em a} and {\em c},  and that for $q=1$ we obtain the classical $SU(2)$.
Notice that the relations
 $ac^{*}=qc^{*}a,$   $ca^{*}=qa^{*}c$, $c^{*}a^{*}=qa^{*}c^{*}$,  $c^{*}c=cc^{*}$, and  $aa^{*}+q^2c^{*}c=a^{*}a+c^{*}c=1$ 
are equivalent to those given before.

We define a {\em Haar functional}  as
a linear functional $h: \mathfrak{A} \to \mathbb{C}$ on a Hopf algebra $\mathfrak{A}$ which is invariant, i.e. it satisfies the condition
$$((id\otimes h) \circ\Delta)(f) = h(f)I = ((h\otimes id) \circ \Delta)(f),$$ for all $f\in \mathfrak{A}$. In the particular case  $\mathfrak{A}=SU_q(2)$, the existence of the Haar functional and the explicit form of it are known \cite{KS}. Let us recall the theory of co-representations on $SU_q(2)$ before we present  the form of this functional. 
Given a co-algebra  $\mathfrak{A}$ and $V$ a vector space, a linear application  $\phi:V\rightarrow V\otimes \mathfrak{A} $ is called a right co-representation of the co-algebra $\mathfrak{A}$ if it satisfies that 
$$(\phi \otimes id)\circ \phi=(id \otimes \Delta)\circ \phi \;\; {\rm and} \;\; (id\otimes \epsilon)\circ \phi=id.$$ In the same way, a linear application  $\phi:V\rightarrow  \mathfrak{A}\otimes V$ is called a left co-representation  if it satisfies that 
$$(  id \otimes \phi)\circ \phi=(\Delta \otimes id  )\circ \phi \;\; {\rm and} \;\; (\epsilon \otimes id )\circ \phi=id.$$ A subspace $W\subseteq V$ is said to be invariant under the right co-representation if $\phi(W)\subseteq W\otimes \mathfrak{A},$ and it is said to be left co-invariant if $\phi(W)\subseteq  \mathfrak{A}\otimes W.$

 The  quantum complex plane is defined as the complex algebra generated by the elements $x,y$ under the relation 
$xy=qyx,$ that is to say the quotient algebra $\mathbb{C}(x,y)/(xy-qyx).$  The  finite dimensional linear spaces of homogeneous polynomials on the  quantum complex variables $x, y$ are natural spaces on which $SU_q(2)$ co-acts. 
 If $\mathbb{C}^2_q$ denotes the quantum complex plane, then $\mathbb{C}^2_q$ is a left and right co-module algebra of 
$SU_q(2)$ with right coaction determined by
$R(x)=x\otimes a+y\otimes c$, $R(y)=x\otimes b+y\otimes d$, see \cite{K}.
 
As it is the case for the classical Lie group $SU(2)$, the spaces of homogeneous polynomials come into play as the unique irreducible co-representations  of 
$SU_q(2)$. These are labeled by the set 
$\frac{1}{2}\mathbb{N}$ for which the corresponding co-representation is denoted by $T^{(l)}$ acting on the polynomials of degree $2l$ whose complex dimension is $2l+1.$ The co-representation $T^{(l)}$ has an associated matrix with entries belonging to $SU_q(2),$ namely 
$$T^{(l)}:=[t_{ij}^{l}]_{2l+1\times 2l+1},$$ 
where the $t_{ij}^{l}$ are polynomials in the generators $a$ and $c$ of the $*$-Hopf algebra. We remind that the general form of the generators  $t_{ij}^{l}$ is given by polynomials involving $a$, $c$ and their adjoints having the following form, see \cite{KS}:
 
\begin{itemize}
\item[i.] $P(-q^2(c^{*}c))a^{m +n}c^{m-n}$ if $m+n \geq 0, n\leq m.$
\item[ii.] $P(-q^2(c^{*}c))a^{m +n} (-qc^{*})^{n-m}$ if $m+n \geq 0, m\leq n.$
\item[iii.] $P(-q^2(c^{*}c)) (-qc^{*})^{m-n}(a^{*})^{-m- n}$ if $m+n \leq 0, n\leq m.$
\item[iv.] $P(-q^2(c^{*}c))c^{n-m}(a^{*})^{-m- n}$ if $m+n \geq 0, m\leq n.$
\end{itemize}
Furthermore, the set $\{t_{ij}^{(l)}: l\in \frac{1}{2}\mathbb{N}\}$ of entries for these irreducible co-representations of $SU_q (2)$ is a orthogonal basis for the underlying vector space of $SU_q(2)$, with respect to the inner product defined by
$\langle y, x \rangle:=h(xy^*)$.  It is known that    $h(t_{ij}^{(l)}.(t_{ij}^{(l)})^*) =[2l+1]_q^{-1}q^{2j}$. Here for any nonzero complex number $q$, and  $x \in \mathbb{C}$,  the $q$-number $[x]_q$ is defined by $[x]_q=\frac{q^x-q^{-x}}{q-q^{-1}}$.
\\ \\From now on we will refer to $SU_q(2)$ as its corresponding  Gelfand-Naimark-Seigel completion, obtained by completing the underlying algebra in the induced topology by the norm  defined by $\|f\|:=(h(ff^{*}))^{\frac{1}{2}}$. 
\\ \\The first ingredient for a Fourier Analysis on this algebra is the concept of the Fourier transform.

\begin{defi}
The {\em Fourier transform} for $f\in SU_q(2)$ is the matrix valued operator  given by $\big(\hat{f}(T^{(l)})\big)_{mn}=h(f(t^{l}_{nm})^{*})$, where $T^{(l)}=[t_{ij}^l]_{-l\leq i,j\leq l}$ denotes an irreducible matrix co-representation. We write $\hat{f}(l):=\hat{f}(T^{(l)})$.
\end{defi}

We have a  Fourier inversion formula in this context and, as a consequence, a decomposition of the algebra of the Peter-Weyl theorem type \cite{MMNNU:k}. Indeed, for a given matrix $A_{n\times n}$ we  define   $$Tr_q(A):=Tr(D_qA)$$ where $D_q:=Diag(q^{-2},..., q^{-2i},..., q^{-2n})$ is a diagonal matrix (we want to point out that, since the algebra is noncommutative, this is not  a trace on the space of matrices with entries on $SU_q (2)$). 
\begin{teo}
Let $f\in SU_q(2)$, then  the {\em Fourier Inversion Formula} is
$$f=\sum_{l\in \frac{1}{2}\mathbb{N}}[2l+1]_qTr_q(\hat{f}(l)T^l ).$$ 
\end{teo}
\begin{proof}
Suppose that $f=\sum c_{ij}^l t_{ij}^l$ is an element of $SU_q(2)$.  Then by the orthogonality properties of the basis elements $t_{ij}^l,$ we have that the complex coefficients are given by $c_{ij}^l=[2l+1]_q q^{-2j}\hat{f}(l)_{ji}.$ This implies that $$f=\sum_{l\in \frac{1}{2}\mathbb{N}} \sum_{ -l\leq i,j\leq l} \left( [2l+1]_q q^{-2j}\hat{f}(l)_{ji}\right) t_{ij}^l.$$ 
Observe that ,  $$ [\hat{f}(l)T^l)]_{jj}=\displaystyle \sum_{-l\leq i \leq l}  \hat{f}(l)_{ji}(T^l)_{ij} \;{\rm and} \;[D_q\hat{f}(l)T^l)]_{jj}=q^{-2
j}\displaystyle \sum_{-l\leq i\leq l}  \hat{f}(l)_{ji}(T^l)_{ij}.$$
Then 
$$f=\sum_{l\in \frac{1}{2}\mathbb{N}}[2l+1]_q Tr_q(\hat{ f}(l)T^l).$$

\end{proof}

We also have a version of the {\em Plancherel's Identity}  \cite{MMNNU:k}.

\begin{teo}
Let $f$ and $g$ elements in $ SU_q(2)$. Then $$h(fg^{*})=\sum_{l\in \frac{1}{2}\mathbb{N}}[2l+1]_qTr_q(  \hat{f}(l)(\hat{g}(l))^{*}).$$
\end{teo}
\begin{proof}
First we suppose that $f=t_{ij}^l$ and $g=t_{rs}^m.$  It is true that $$[\hat{f}(r)]_{wt}=[\widehat{t_{ij}^l}(r)]_{wt}=h(t_{ij}^l (t_{tw}^r)^{*})=\delta_{lr}\delta_{it}\delta_{jw}[2l+1]_q^{-1}q^{-2j}.$$ For this reason $[\widehat{t_{ij}^{l}}(l)  \left(\widehat{t_{ij}^{l}}\right)^{*}(l)]_{kk}=\delta_{ik}\delta_{ij} [2l+1]_q^{-2}q^{-4j}.$ Then we have that $$h(fg^{*})=\delta_{ir}\delta_{rs}\delta_{lm}[2l+1]_q Tr_q( \hat{f}(l)(\hat{g}(l))^{*}),$$ where we used Kronecker delta functions.The proof is complete if we consider that the set $\{t_{ij}^l\}$ is a basis for $SU_q(2).$
\end{proof}

\section{Representations of   $SU_q(2)$ on an Algebra of Periodic Pseudo-differential Operators } 
 
 In this section we construct a $*$-representation of $SU_q(2)$ as bounded pseudo-differential operators acting on $L^2(\mathbb{S}^1)$ and introduce the concept of especial unitary representation space in order to emulate a background manifold for $SU_q(2)$. We start  by recalling the following statement about the classification theorem for all the  $*$- representations of $SU_q(2)$ on bounded operators acting  on Hilbert spaces \cite{W}.

\begin{teo}[Woronovicz] \label{Woro} 
For any  complex number $u\in \mathbb{S}^1$ there are two irreducible $*$-representations $\pi_u^1$ and $\pi_u^\infty$ of $SU_q(2)$ with dimensions $1$ and infinite, respectively. These are described as follows:
\begin{itemize}
\item[i.] The action on the one dimensional linear space $\mathbb{C}$ is given by $\pi_u^1(a)z=uz,$ and $\pi_u^1(c)z=0,$ for all $z\in \mathbb{C}.$
\item[ii.]For $\{e_n: n\in \mathbb{N}\}$ an orthonormal basis for a Hilbert space $H,$ we have the action, $\pi_u^\infty(a)e_n=\sqrt{1-q^{2n}}e_{n-1},$ and the element c is represented by  $\pi_u^\infty(c)e_n=q^{n}ue_n,$ where $e_{-1}=0.$
\end{itemize}
\end{teo}
It is easy to see that  
$\left(\pi_u^\infty(a)\right)^{*}(e_n)=\pi_u^\infty(a^*)e_n=\sqrt{1-q^{2n+2}}e_{n+1},$ and also that the adjoint operator of $\pi_u^\infty(c)$ is given by $\pi_u^\infty(c^*)e_n=q^{n}\overline{u}e_n.$\\
 The next step in our construction is  to choose and order  the canonical basis for $L^2(\mathbb{S}^1)$ in a suitable way. Indeed, if we denote $e^{-in\theta}:=e_{2n}$ and $e^{in\theta}:=e_{2n-1}$  for each $n\in \mathbb{N}$, then the set $(e_n)_{n\in \mathbb{N}}$ is now an ordered orthonormal basis.\\ 
 
Take a fixed element $\nu \in \mathbb{S}^1.$ Let $\sigma_{c,\nu}: \mathbb{S}^1\times \mathbb{Z} \rightarrow \mathbb{C}$ be the function defined by $\sigma_{c,\nu}(\theta, n)=q^{-2n}\nu$ for  each element  $(\theta,n)\in \mathbb{S}^1 \times  \mathbb{Z}^{-}\cup \{0\} $ and $\sigma_{c,\nu}(\theta, n)=q^{2n-1}\nu$  for $(\theta,n)\in \mathbb{S}^1 \times  \mathbb{Z}^{+} .$ This function is the symbol of the global pseudo-differential operator on the circle (we follow the terminology of (\cite{R.andT}, \cite{R.Tbook}) $$T_{c, \nu}f(\theta)=\sum_{n\in \mathbb{Z} } \hat{f}(n)\sigma_{c,\nu}(\theta, n)e^{i\theta n},$$ acting on functions $f\in L^1(\mathbb{S}^1),$ where $\hat{f}(n)$ represents the Fourier transform of  the periodic function $f$, given by $\hat{f}(n)=\int_{\mathbb{S}^1}f(\theta)e^{-i\theta}d\theta.$
\\Similarly, let $\sigma_{a,\nu}: \mathbb{S}^1\times \mathbb{Z} \rightarrow \mathbb{C}$ be the function defined  in the following way:  $\sigma_{a,\nu}(\theta, 0)=0,$ $\sigma_{a,\nu}(\theta, n)=\sqrt{1-q^{2(2n-1)}}e^{-(2n+1)i\theta}$ for $n\in \mathbb{Z}^{+},$
 and $\sigma_{a,\nu}(\theta, n)=\sqrt{1-q^{2(-2n)}}e^{-2n i\theta}$ for $n\in \mathbb{Z}^{-}.$
 This function is the symbol of the pseudo$-$differential operator on the circle $$T_{a,\nu}f(\theta)=\sum_{n\in \mathbb{Z} } \hat{f}(n)\sigma_{a,\nu}(\theta, n)e^{i\theta n}.$$
 
 \begin{teo}
For each $\nu \in \mathbb{S}^1$, the  global pseudo-differential operators $T_{c, \nu}$ and $T_{a, \nu}$ satisfy the properties of Woronowicz's theorem.

\end{teo}

\begin{proof}
Indeed, for each $n\in \mathbb{N}$, a straightforward computation shows that:
 
 \begin{itemize}
 \item[] $T_{c, \nu}(e_{2n})=T_{c,\nu}(e^{-in\theta})=\sigma_{c,\nu}(\theta, -n)e^{-in\theta}=q^{2n}\nu e^{-in\theta}=q^{2n}\nu e_{2n}$,
 
 \item[] $T_{c, \nu}(e_{2n-1})=T_{c,\nu}(e^{in\theta})=\sigma_{c,\nu}(\theta,n)e^{in\theta}=q^{2n-1}\nu e^{in\theta}=q^{2n-1}\nu e_{2n-1}$,
 
 \item[] $T_{a,\nu}(e_{2n})=T_{a,\nu}(e^{-in\theta})=\sigma_{a,\nu}(\theta, -n)e^{-in\theta}=\sqrt{1-q^{2(2n)} }e^{i2n\theta}e^{-in\theta}$  $=\sqrt{1-q^{2(2n)} }e^{in\theta}=\sqrt{1-q^{2(2n)} } e_{2n-1}$, 
 
 \item[] and
 
 \item[] $T_{a,\nu}(e_{2n-1})=T_{a,\nu}(e^{in\theta})=\sigma_{a,\nu}(\theta, n)e^{in\theta}=\sqrt{1-q^{2(2n-1)} }e^{-(2n+1)i\theta}e^{in\theta}$
 \\ $=\sqrt{1-q^{2(2n-1)} }e^{(-n-1)\theta}=\sqrt{1-q^{2(2n-1)} } e_{2n-2}.$
 \end{itemize}
 \end{proof}

Thus, by theorem \ref{Woro}, the  pseudo-differential operators $T_{c,\nu}$ and $T_{a,\nu}$ generate a unique $C^{*}$-algebra which is a representation of $SU_q(2)$ on the Hilbert space $L^2(\mathbb{S}^1)$, for each $\nu \in \mathbb{S}^1$. 

\begin{rem}
An interesting analytical property is that all the elements of this algebra  are bounded as operators from $L^p(\mathbb{S}^1)$ to  $L^p(\mathbb{S}^1)$.
\end{rem}

 In the existent literature on the quantum group $SU_q(2)$  it is only described by mean of the generators of the algebra. As a noncommutative algebra, this algebra does not corresponds to the algebra of functions on a manifold but, on the lines of noncommutative geometry \cite{Connes}, it is interpreted as the algebra of functions of some `noncommutative background manifold".  In the aim of keeping track with the natural comparison with $SU(2)$ we introduce the following representation.
 
\begin{defi}
Let $q\in (0,1)$ be a real number and $H$ be an infinite dimensional Hilbert space on which $SU_q(2)$ acts according to the Woronowicz's theorem  (c.f. theorem \ref{Woro}). We define the {\em especial unitary representation space of} $SU_q(2)$ as the set $$SU^{H}_q(2)=\{X_z : z \in \mathbb{S}^1 \}\subset M_{2\times 2}(\mathbb{C})\otimes B(H) ,$$ where $X_z=
 {\begin{bmatrix}
\pi_z^\infty(a) &-q\pi_z^\infty(c^*) \\
\pi_z^\infty(c) &\pi_z^\infty(a^*) \\
 \end{bmatrix} } \,,$ and $B(H)$ represents the Banach algebra of bounded linear operators acting on $H.$
\end{defi}
 From this point of view the background space corresponding to the coordinate algebra $SU_q(2)$ is the one behind this matrix representation of operators parametrized by a circle. The composition in this representation reminds the usual one for $2 \times 2$ matrices in classical $SU(2)$:

\begin{pro}
Let $q\in (0,1),$ be a real number and $z,$ and $z'$ be complex numbers on the circle. Consider the extension of the matrix product to the the space $SU^{H}_q(2)$. Then, for any pair $X_z$ and $ X_{z'} $  $\in SU^{H}_q(2) $, we have  
\[X_z X_{z'} =
\left[ {\begin{array}{cc}
X_{11}&-X_{21}^* \\
X_{21}&X_{11} ^*\\
 \end{array} } \right]\,,
\]
where $$X_{11}(e_n)=\pi_{zz'}^\infty(a)(\sqrt{1-q^{2n}}e_{n-1}-\frac{\bar{z}z' q^{2n+1}}{\sqrt{1-q^{2n+2}}} e_{n+1})$$ and
 $$X_{21}(e_n)=\pi_{zz'}^\infty(c) (\sqrt{1-q^{2n}}\bar{z'}e_{n-1}+\frac{\sqrt{1-q^{2n+2}}\bar{z}}{q}e_{n+1}).$$

\end{pro}

\begin{proof}
We have that  $X_z X_{z'}=$  \[ 
\left[ {\begin{array}{cc}
\pi_z^\infty(a)\pi_{z'}^\infty(a)+ (-q)\pi^\infty_z(c^{*})\pi^\infty_{z'}(c) &\pi^\infty_z(a) (-q)\pi^\infty_{z'}(c^{*})+(-q)\pi_z^\infty(c^{*} )\pi^\infty_{z'}(a^{*})\\
\pi_z^\infty(c)\pi_{z'}^\infty(a)+ \pi^\infty_z(a^{*})\pi^\infty_{z'}(c)&  \pi_z^\infty(c)(-q)\pi_{z'}^\infty(c^{*})+ \pi^\infty_z(a^{*})\pi^\infty_{z'}(a^{*})\\
 \end{array} } \right]\,.
 \]  
Now, from the explicit form of the representation we can see that $$\big( \pi_z^\infty(a)\pi_{z'}^\infty(a)+ (-q)\pi^\infty_z(c^{*})\pi^\infty_{z'}(c)\big)(e_n)=\sqrt{1-q^{2n}}\sqrt{1-q^{2n-2}}e_{n-2}+(-q)q^{2n}\bar{z}z' e_n$$ $$=\pi_{zz' }^\infty(a)(\sqrt{1-q^{2n}}e_{n-1}-\frac{\bar{z}z' q^{2n+1}}{\sqrt{1-q^{2n+2}}} e_{n+1}),$$
and $$\big(\pi_z^\infty(c)\pi_{z'}^\infty(a)+ \pi^\infty_z(a^{*})\pi^\infty_{z'}(c)\big)(e_n)=q^{n-1}\sqrt{1-q^{2n}}ze_{n-1}+q^n z'\sqrt{1-q^{2n+2}}e_{n+1}$$ $$=\pi_{zz'}^\infty(c) (\sqrt{1-q^{2n}}\bar{z'}e_{n-1}+\frac{\sqrt{1-q^{2n+2}}\bar{z}}{q}e_{n+1}),$$ which ends the proof.
 \end{proof}

  \begin{rem}
Using the already defined concept of special unitary  representation space $SU^H_q(2)$, we can think on the elements of the algebra $SU_q(2)$ as "functions" on this space with values in $B(H)$ in the following way: for the generator $c\in SU_q(2)$ and $x\in \mathbb{S}^1$ we define $c(x)=\pi^\infty_x(c),$ and similarly $a(x)=\pi^\infty_x(a).$ In this sense, for $f \in SU_q(2)$,  $f(x):=\pi^\infty_x(f),$ after the identification  $X_x\equiv x$ for $X_x\in SU^H_q(2)$. 
\end{rem}

We have a version for $SU_q^H(2)$ of the regular representation of a compact Lie group $G$ on $L^2(G)$, via the parametrization of $SU^H_q(2)$   by the compact Lie group $\mathbb{S}^1$. 
\begin{pro}\label{A2}
Define the map $\phi:\mathbb{S}^1\rightarrow \mathfrak{U}(SU_q(2))$ by $\phi(v)=\phi_v$ where $$\phi_v(f)(x):=f(xv)$$ for $v,x\in \mathbb{S}^1$, and $f \in SU_q(2)$.  
Then this map induces a  unitary representation of $SU^H_q(2)$ via the parametrization  by the compact Lie group $\mathbb{S}^1$.
\end{pro}
\begin{proof}
First, it is clear that $$\phi_{uv}(f)(x):=f(xuv)=(\phi_{v})\circ (\phi_{u})(f)(x).$$ In addition,  $\phi_1(f)(x):=f(x),$ thus 
$\phi_1:=\phi(1)=Id_{SU_q(2)}.$
\\Now we show that $\phi_v:=\phi(v)\in \mathfrak{U}(SU_q(2)).$ Indeed, from the action on the generators $a,c,a^{*}, c^{*}$, and the form of $t_{ij}^l$ reminded before, we have
that $\phi_v (a)=a$, $\phi_v (a^{*})=a^{*}$, $\phi_v(c)= vc$ and $\phi_v(c^{*})= \bar{v}c^{*}$. From this we conclude that $$\langle \phi_v(t_{ij}^{(l)}), (t_{ij}^{(l)})\rangle=h(\phi_v(t_{ij}^{(l)}) (t_{ij}^{(l)})^{*})=h(t_{ij}^{(l)} \phi_v(t_{ij}^{(l)})^{*})=\langle  (t_{ij}^{(l)}) , \phi_v(t_{ij}^{(l)}) \rangle.$$
\end{proof}

Let us mention the following  result on the compactness of $SU^H_q(2)$.

\begin{teo}
For every $q\in (0,1)$ we have that $SU_q(2)^H$ is a compact subset of the topological space $M_{2\times 2}(\mathbb{C})\otimes\mathbb{B}(H)$ considering the product topology in this space.
\end{teo}

\begin{proof}
Let $q\in (0,1)$ be a given real number. Consider the natural map $\phi_q: \mathbb{S}^1 \rightarrow M_{2\times 2}(\mathbb{C})\otimes\mathbb{B}(H),$  which acts as $\phi_q(u)=X_u$ where \[X_u=
\left[ {\begin{array}{cc}
\pi_u^\infty(a) &-q\pi_u^\infty(c^*) \\
\pi_u^\infty(c) &\pi_u^\infty(a^*) \\
 \end{array} } \right]\,.
\] This map is easily seen to be continuous, then the desired conclusion is obtained.
\end{proof}

. It is very often mentioned that, for $0<q<\nu<1$, $SU_q(2)$  is ``less" commutative than $SU_\nu(2)$. We end this section with the following theorem, which gives a hint ---in terms of the infinite-dimensional representations introduced before--- to distinguish them in terms of algebraic properties of the parameters.

\begin{teo}
Let H be a separable Hilbert space on which \su  is represented. Consider the sub-algebra $\langle c\rangle _\mathbb{Q}$ of \su generated by the generating element $c$ with coefficients in the field $\mathbb{Q}+i \mathbb{Q}$.  Then we have that $\pi_z^\infty( \langle c\rangle _\mathbb{Q} )\subseteq Aut(H)$ for each  $z\in \mathbb{S}^1$ if and only if $q$ is a transcendental number.
 \end{teo}
  
 \begin{proof}
 Suppose $q\in (0,1)$  is an algebraic number, then there exists an irreducible polynomial with rational coefficients $P(x)=r_0+r_1x+...+r_nx^n$ such that $P(q)=0.$ It is clear that the element $e_1\in H$ in the ordered basis for the Hilbert space H is in the kernel of the operator $\pi_u^\infty(r_0+r_1c+...+r_{n}c^n)$, thus 
 $\pi_u^\infty( \langle c\rangle _\mathbb{Q} )\subseteq Aut(H)$.
 \end{proof}

\section{Global Pseudo-differential Operators on $SU_q(2)$}\label{S:PDOsSUq(2)} 
In this section we define a global pseudo-differential calculus for the quantum group $SU_q(2)$ on the lines of \cite{RSM}. We will consider the Fourier order of the corresponding symbols and provide  a composition formula for the algebra of global pseudo-differential operators graded with respect to such Fourier order.\\ \\
In \cite{RSM}, the authors consider the quantum group $SU_q(2)$ in the spirit of Noncommutative Geometry \cite{Connes 2}, where the spectral geometry of a noncommutative algebra is studied by  means of a reference operator (a Dirac-type operator). In particular, from a summable operator $D_\mathcal{A} : L^2(\mathcal{A}) \to L^2(\mathcal{A}) $ defined by a sequence
of eigenvalues according to the Peter-Weyl decomposition of the quantum group $\mathcal{A}$, and subject to a summability
condition, they define a smooth domain $C^\infty_{D_\mathcal{A}}:=\bigcap_\alpha Dom(|D_\mathcal{A}|^\alpha)$   in terms of which they obtain a``bare" spectral triple 
$(C^\infty_{D_\mathcal{A}} , L^2(\mathcal{A}), D_\mathcal{A}) $. The notion of full symbol of the compact Lie group case \cite{R.Tbook} goes through and, based on the Fourier theory for compact quantum groups, the following notion of global pseudo-differential operator is given in the case of $\mathcal{A}=SU_q(2)$.

\begin{defi}
A linear continuous operator $A:C^\infty_D\rightarrow [C^\infty_D]^{*}$ is called a pseudo differential operator. If the corresponding Schwartz kernel $K_A$ satisfies in particular that $K_A\in C^\infty_D\hat{\otimes}C^\infty_D$ then the pseudo-differential operator is called regular.
\end{defi}

The statement of the theorem 6.14 in \cite{RSM} is exactly the definition of pseudo-differential operator in the present document. The motivation of the authors of \cite{RSM}, as well as the motivation of this work, is to adopt the definition from the theory exposed in  chapter \cite{R.Tbook} for the case of compact Lie groups.

\begin{defi} A {\em global pseudo-differential operator} on  $SU_q(2)$ is a linear operator that can be written in the form 
$$T_\sigma f=\sum_{l\in \frac{1}{2}\mathbb{N}}[2l+1]_qTr(\sigma(l)\hat{f}(T^{l})T^{l}),$$ where  the function $\sigma: \frac{1}{2}\mathbb{N} \rightarrow \bigcup_{l} M_{2l+1\times 2l+1}(\mathbb{C})\otimes SU_q(2)$ is the {\em symbol} of the operator. Equivalently, a pseudo-differential operator on $SU_q(2)$ can be written in the form $T_\sigma f= \mathfrak{F}^{-1}(\sigma \mathfrak{F}(f) )$ where  $\mathfrak{F}$ and $\mathfrak{F}^{-1}$ are the Fourier and its inverse operators.
\end{defi}

Notice that, if  we consider a symbol $\sigma$ with entries in the algebra $SU_q(2)$, then we can think that $\sigma=\sigma(x,l)$ depends on the ``space" variable $X_x\in SU^H_q(2)$. In this sense, this definition extends in a formal way the concept of a global pseudo-differential operators on the compact Lie group $SU(2)$.
\\�\\
We start our analysis of the properties of the pseudo-differential operators on \su \, with the following Lemma, which establishes a necessary condition for a linear operator to be a pseudo-differential operator. We want to recall that  the set given by $\{t_{ij}^l: l\in\mathbb{N} \, \; {\rm and } \,\; -l\leq i,j\leq l\}$ is the basis provided by the irreducible co-representations introduced in the section \ref{S:FAonSUq(2)} before. The following important lemma is obtained after considering 
the form of the Fourier transform of the elements $t_{ij}^l$.
\begin{lema}\label{Key}
Let  $\sigma $ be a symbol and $T_\sigma$ be its corresponding global pseudo-differential operator acting on \su. Then $$T_\sigma (t_{ij}^{l})=\displaystyle\sum_{k=0}^{2l} \sigma_{l-k, j}(l)t_{i,k-l}^l,$$
where $\sigma_{l-k, j}$ denotes the corresponding entry in the matrix symbol $\sigma$.
\end{lema}
 As corollary we obtain 
 \begin{coro}
 Let  $\sigma $ be a symbol with  complex entries,  and let $T_\sigma$ be its corresponding pseudo-differential operator acting on \su. Then $$\|T_\sigma(f)\|^2=\sum_{l}\displaystyle\sum_{k=0}^{2l} (\sigma_{l-k, j})^2 [2l+1]_q^{-1}q^{-2(k-l)}(f_{i, k-l}^l)^2,$$
 where $f=\sum f_{ij}^l t_{ij}^l\in SU_q(2)$, for $ f_{ij}^l \in \mathbb{C}$, and the norm is the defined in terms of the Haar functional as $\|f\|^2:=h (ff^*)$ for all $f\in SU_q(2)$.
 \end {coro}
 \begin{proof}
 Let $\sigma $ be a symbol satisfying the conditions of the theorem. Take $f\in SU_q(2),$ then $f=\sum f_{ij}^l t_{ij}^l$ and  we have that, by Plancherel's Identity,
 $$\|T_\sigma(f)\|^2=\sum_{l}[2l+1]_qTr\left(\widehat{T_\sigma(f)}(l) \widehat{T_\sigma(f)}^{*}(l)\right).$$
 We can see that $[\widehat{T_\sigma(f)}(l)]_{k-l,i}=\sigma_{l-k,j}f_{i,k-l}^{l}[2l+1]_{q}^{-1}q^{-2(k-l)}$, thus  $$\|T_\sigma(f)\|^2=\sum_{l}\displaystyle\sum_{k=0}^{2l} (\sigma_{l-k, j})^2 [2l+1]_q^{-1}q^{-2(k-l)}(f_{i, k-l}^l)^2.$$
 \end{proof}

\begin{rem}
Let $A: SU_q(2)\rightarrow SU_q(2) $ be a global pseudo-differential operator, then   $A$ is invariant with respect to the regular representation defined before  in the following sense: 
$$A(\phi_v f)(u):=\phi_v(Af)(u)$$ where $u,v\in \mathbb{S}^1,$ and $f\in SU_q(2)$. Here  $\phi_v f(u):=f(uv)$ as defined in the proposition \ref{A2}.
\end{rem}

There is a notion of {\em order} for global pseudo-differential operators on compact Lie groups which is perfectly adapted to compare symbol/operator classes with the corresponding symbol/operator classes in the local pseudo-differential calculus of H{\"o}rmander. It is also used, among many other things, to give conditions on the symbols for those operators to be bounded in Sobolev spaces (see theorems 10.8.1 and  10.9.6 in \cite{R.Tbook}). 
We introduce a different notion of order, that of {\em Fourier order} and {\em Fourier classes} of symbols for global pseudo-differential operators on $SU_q(2)$. 

\begin{defi}
A function $$\sigma: \frac{1}{2} \mathbb{N}\rightarrow \bigcup_{l\in \frac{1}{2} \mathbb{N}} M_{2l+1\times 2l+1}(\mathbb{C})\otimes SU_q(2)$$ will be called a  {\em homogeneous  symbol of Fourier order} $m\in \frac{1}{2} \mathbb{N}$ if  the following conditions are satisfied:
\begin{itemize}
\item[i.] For each $l \in \frac{1}{2} \mathbb{N}$ there exists a map  $\psi_{\sigma}(l): I_{2l+1}\times I_{2l+1}\rightarrow I_{2l+1}\times I_{2l+1}$ where $I_{2l+1}:=\{\-l,-l+1,..., l\}$, and  the entries of $\sigma$ satisfy  $\sigma(x,l)_{i,j}\in \textsf{Span}\{t_{\psi_\sigma(l)(i,j)}^m\}$.
\item[ii.]If $\psi_\sigma(l)(i,j)\notin  I_{2m+1}\times I_{2m+1}$ then $\sigma(x,l)_{i,j}=0$.
\end{itemize}
\end{defi}
Now, if  $\sigma=\sum c_k\sigma_k$ with $\sigma_k$ homogeneous of Fourier order $m$ and $c_k\in \mathbb{C}$, then we say that $\sigma $ is a  symbol of Fourier order $m\in \frac{1}{2} \mathbb{N}$. The class of symbols of order $m\in \frac{1}{2} \mathbb{N}$ is denoted by $\Phi^m(SU_q(2))$. 
We will say that $\sigma$ is a {\em homogeneous symbol of negative Fourier order} $-m$, for $m\in \frac{1}{2} \mathbb{N}$,  if there exists a homogeneous symbol $\beta \in \Phi^m(SU_q(2))$ such that  $T_\beta \circ T_\sigma= T_\sigma \circ T_\beta=Id$. 
If  $\sigma=\sum c_k\sigma_k$ with $\sigma_k$ homogeneous of Fourier order $-m$ and $c_k\in \mathbb{C}$, then we say that $\sigma $ is a  symbol of Fourier order $-m$. The class of symbols of order $-m$ is denoted by $\Phi^{-m}(SU_q(2))$. 
We say that an operator has Fourier order $m\in \frac{1}{2}\mathbb{Z}$  if the corresponding symbol  belongs to the class  $\Phi^m(SU_q(2))$.  
Finally, the principal symbol of $\sigma \in \Phi^m(SU_q(2))$ is the sum of the homogeneous symbols of higher Fourier order in the decomposition of $\sigma$ into homogeneous symbols. \\�\\
In order to see that the symbol classes $\Phi^m(SU_q(2))$, $m\in \frac{1}{2}\mathbb{Z}$ form a graded algebra  we will prove a composition formula for operators with symbols of Fourier order zero  first and later the corresponding result for operators of positive Fourier order.

\begin{teo}[Composition Formula I]
Let  $A$ and $B$ be pseudo-differential operators on $SU_q(2)$ with symbols  $\sigma_A\in \Phi^0(SU_q(2)) $, for  and $\sigma_B\in \Phi^0(SU_q(2))$ respectively.  Then  the  composition operator $A\circ B$ is a pseudo-differential operator whose symbol is the product of the symbols, i.e.  $\sigma_{A \circ B}(l)=\sigma_A(l) \sigma_B(l)$ for all $l\in \frac{1}{2}\mathbb{N}$.
\end{teo}

\begin{proof}
Let  $\sigma_A\in \Phi^0(SU_q(2)) $ and $\sigma_B\in \Phi^0(SU_q(2))$. Then all the entries of the symbol $\beta$ are complex numbers and, taking into account the lemma \ref{Key}, $\widehat{B(t_{ij}^l)}=\sigma_B \widehat{t_{ij}^l}$. Thus $$A\circ B (t_{ij}^r)=\sum_l[2l+1]_qTr(\sigma_A \sigma_B\widehat{t_{ij}^r}(l)T^l),$$ implying that $\sigma_{A \circ B}(l)=\sigma_A(l) \sigma_B(l)$ for all $l\in \frac{1}{2}\mathbb{N}$. 
\end{proof}

\begin{rem}Notice that the theorem above is also true for $\sigma_A\in \Phi^m(SU_q(2)) $ for $m\in \frac{1}{2}\mathbb{N}$.
\end{rem}

As another consequence of the lemma \ref{Key}, we have a formula relating the symbol of an operator of Fourier order zero with the operator action on the basis elements $t_{ij}^l$.
Actually, the symbol of a pseudo-differential  operator {\em A} of Fourier order zero satisfies that 
$$\sigma_{l-k,j}(l)=[2l+1]_q q^{-2(k-l)}h(T_\sigma(t_{ij}^l)(t_{i,k-l}^l )^{*})),$$
where $0\leq k\leq 2l$.

\begin{exa}
Many of the important linear operators already defined in the study of the quantum group $SU_q(2)$ are global pseudo-differential operators. Indeed, for a linear operator $A$  on $SU_q(2)$ such that $A(t_{ij}^l)=\lambda(l)t_{ij}^l$, where $\lambda$ is a complex valued function, then $A$ is a global  pseudo-differential operator with diagonal symbol $\sigma_A(l)=\lambda(l)I_{2l+1, 2l+1}\in \Phi^0(SU_q(2).$ In consequence,  for the particular cases of $\lambda(l)=2l+1,$ or $\lambda(l)=[l]_q[l+1]_q$, the true and naive Dirac Operators defined in \cite{VS} , \cite{Connes} and \cite{Goswami} are also  global pseudo-differential operators of Fourier order zero. 
\end{exa}

Recall that $$t^m_{rs}t^n_{ij}=\sum_{|n-m|\leq p\leq n+m}C(n,m,p ;r,s,i,j)t^p_{i+r,j+s},$$
defines the {\em Clebsch-Gordan coefficients}, and the matrix of these coefficients is invertible, see \cite{KS}. 
Let $\sigma \in \Phi^{m}(SU_q(2))$  be a homogeneous symbol and $\beta \in \Phi^{n}(SU_q(2))$ be any symbol. Since $\sigma$ is homogeneous, by definition,  r each $l \in \frac{1}{2} \mathbb{N}$ there exists a map  $\psi_{\sigma}(l): I_{2l+1}\times I_{2l+1}\to I_{2l+1}\times I_{2l+1}$ where $I_{2l+1}:=\{\-l,-l+1,..., l-1,  l\}$, and  the entries of $\sigma$ satisfy that $\sigma(x,l)_{i,j}\in \textsf{Span}\{t_{\psi_\sigma(l)(i,j)}^m\}$. In addition, if $\psi_\sigma(l)(i,j)\notin  I_{2m+1}\times I_{2m+1}$ then $\sigma(x,l)_{i,j}=0$, where we write $\psi_\sigma:=(\psi_{\sigma}^1  ,\psi_{\sigma}^2)$. In order to know the entries of the symbol of the composition operator $T_\gamma= T_\beta\circ T_\sigma $ it is enough to compute the image of the basis elements. We have, using lemma \ref{Key}, that
  $$T_\gamma(t_{i,j}^{l})=[2l+1]_qT_\beta( \sum_{0\leq k\leq 2l}\sigma(l)_{l-k,j}t_{i,k-l}^{l})$$ $$=[2l+1]_q\sum_{0\leq k\leq 2l} \sum_{|l-m|\leq p\leq l+m}C_I \left(   \sum_{0\leq d \leq 2p} \beta(p)_{p-d, \psi_{\sigma}^2(l-k,j)+k-l} \,\,\,t_{\psi_{\sigma}^1(l-k,j)+i,d-p}^p\right) ,$$ where $C_I:=C_I(l,m,p;\psi_{\sigma}^1(l-k,j), \psi_{\sigma}^2(l-k,j))$ are the Clebsch-Gordan coefficients.
One can see that the principal symbol appears for $p=l+m.$ To compute this principal symbol we proceed in the following way:  we put $p=l+m$ in the above series and we obtain
 $$[2l+1]_q\sum_{0\leq k\leq 2l}C_I\left( \sum_{0\leq d \leq 2(l+m)} \beta(l+m)_{l+m-d,\, \psi_{\sigma}^2(l-k,j)+k-l}\,\,\, t_{\psi_{\sigma}^1(l-k,j)+i,\,d-l-m}^{l+m}\right).$$
Now taking into account the definition of {Clebsch-Gordan} coefficients \cite{V}, we can decompose the last element as follows
 $$ t_{\psi_{\sigma}^1(l-k,j)+i,\,\,d-l-m}^{l+m}=\sum_{u+s=d-l-m}C_Jt_{\psi_{\sigma}^1(l-k,j),u}^{m}\,\,\,t_{i,s}^l,$$
 where $C_J$ depends on the Clebsch-Gordan coefficients. From the computations above we can see that the principal symbol is given by
 $$\frac{1}{[2l+1]_q}\gamma(l)_{-s,j}=$$ $$\sum_{0\leq k\leq 2l}C_I  \sum_{0\leq d \leq 2(l+m)} \beta(l+m)_{l+m-d, \psi_{\sigma}^2(l-k,j)+k-l} \,\,C_Jt_{\psi_{\sigma}^1(l-k,j),d-l-m-s}^{m}.$$
Thus, we have proved 
 
 \begin{teo}\label{Composition Formula II}
Let $\sigma \in \Phi^{m}(SU_q(2))$ be a homogeneous symbol and $\beta \in \Phi^{n}(SU_q(2))$ be any symbol.
Then the principal symbol $\gamma $ of the pseudo-differential operator 
$T_\sigma \circ T_\beta$ satisfies that:
$$\frac{1}{[2l+1]_q}\gamma(l)_{-s,j}=$$ $$\sum_{0\leq k\leq 2l}C_I  \sum_{0\leq d \leq 2(l+m)} \beta(l+m)_{l+m-d, \psi_{\sigma}^2(l-k,j)+k-l} \,\,C_Jt_{\psi_{\sigma}^1(l-k,j),d-l-m-s}^{m},$$
where $\psi_\sigma(l)(i,j):=(\psi_1(i,j),\psi_2(i,j))$.
\end{teo}
 
\begin{exa}
The multiplication operator $M_{t_{ij}^m}$ by a fixed element basis $t_{ij}^m\in SU_q(2)$ is a global pseudo-differential operator of order $m$ with diagonal symbol $\sigma_{M_{t_{i,j}^{m}} }(l)=t_{i,j}^{m}I_{2l+1\times 2l+1}$. The composition $M_{t_{ij}^{m_1}}\circ M_{t_{rs}^{m_2}}$  of two multiplication operators is a global pseudo-differential operator of Fourier order $m_1+m_2$. 
\end{exa}

The following theorem shows that our definition of order for a symbol is in fact an order in the sense that we obtain a graded algebra of pseudo-differential operators.
In \cite{RSM} it has been  proved that any linear operator on $SU_q(2)$  is a pseudo-differential operator in sense adopted in this thesis. Using this fact we can prove the following 
\begin{teo}
Let $\sigma $ and $\beta$ be symbols in $\Phi^{k_1}(SU_q(2))$ and $\Phi^{k_2}(SU_q(2))$, respectively, for $k_1, k_2\in \frac{1}{2}\mathbb{Z}$. Then  there exists a symbol $c\in \Phi^{k_1+k_2}(SU_q(2)) $ such that $T_\sigma \circ T_\beta=T_c.$
 
\end{teo}
\begin{proof}
We proceed considering different cases, depending on the signs of the orders and we suppose, without loss of generality, that $\sigma $ and $\beta$ are homogeneous.
\begin{itemize}
\item[i.] Case $k_1, k_2\geq 0:$ This case has been already proved in the theorem \ref{Composition Formula II}.
 \\ \item[ii.] Case $k_1<0$ and $k_2<0:$  By definition $T_\sigma$ and $T_\beta$ are both invertible with inverse operators $(T_\sigma)^{-1}$ and $(T_\beta)^{-1}$ with symbols of orders $-k_1>0$ and $-k_2>0.$ Clearly, as a consequence of theorem \ref{Composition Formula II}, the operator $(T_\beta)^{-1}\circ (T_\sigma)^{-1}$ belongs to $\Phi^{-(k_1+k_2)}(SU_q(2))$ and it is the inverse of the operator $T_\sigma \circ T_\beta$. 
\\  \item[iii.]  Case  $k_1>0$ and $k_2<0,$ and $k_1+k_2>0:$ Let $T_\sigma \circ T_\beta=R.$ Then we have that  $T_\sigma=R(T_\beta)^{-1},$ and this implies that $R$ must be a pseudo-differential operator of order less or equal to $k_1+k_2.$
\\  \item[iv.]  Case  $k_1>0$ and $k_2<0,$ and $k_1+k_2<0:$ Let $T_\sigma \circ T_\beta=R.$ Then we have that  $T_\beta=R(T_\alpha)^{-1},$ and this implies, taking into account item ii., that $R$ is a pseudo-differential operator of order less or equal to $k_1+k_2$. 
\end{itemize}
\end{proof} 

We finish this section with another  important aspect for a complete pseudo-differential calculus,  the adjoint operator of a pseudo-differential operator. We want the equality $\langle T_\sigma(t_{ij}^l), t_{rs}^m\rangle=\langle t_{ij}^l, (T_\sigma)^{*}(t_{rs}^m)\rangle$ to hold, and this implies that, for $(T_\sigma)^{*}:=T_\beta$,
$$\langle T_\sigma(t_{ij}^l), t_{rs}^m\rangle =\sum_{k=0}^{2l} h(\sigma_{ l-k,j}(l)t_{i,k-l}^l(t_{rs}^m)^{*})=\sum_{k=0}^{2m} h(t_{ij}^l (t_{r,k-l}^m)^{*}(\beta_{l-k,s}(m)^{*}))$$ $$=\langle t_{ij}^l, (T_\sigma)^{*}(t_{rs}^m)\rangle.$$
Now, if we take $\sigma \in  \Phi^p(SU_q(2)$ a homogeneous symbol of Fourier order $p\in \frac{1}{2}\mathbb{N}$, then the expression above is zero unless $|l-p|\leq m\leq l+p$. This implies that,  unless $|m-p| \leq l\leq m+p$, we must have $ \langle t_{ij}^l, (T_\sigma)^{*}(t_{rs}^m)\rangle=0 $. We conclude that the following statement holds.

\begin{teo}
Let $T_\sigma \in \Phi^{p}(SU_q(2))$ then the adjoint operator  of $T_\sigma$ is pseudo-differential operator of order  p.
\end{teo}

\section{Final comments on the spectral properties of global pseudo-differential pperators on $SU_q(2)$ }

In this section we use the theory developed in section \ref{S:PDOsSUq(2)}  to obtain several results concerning  spectral properties of very particular types of (symbols for) global  pseudo-differential operators with symbols in the classes $\Phi^m(SU_q(2))$. Let us begin by the following direct consequence of lemma \ref{Key}:
\begin{coro} 
 Let  $\sigma \in \Phi^m(SU_q(2))$   such that $\sigma(l)=0$ for $l> N$  for some $N\in \mathbb{N}$, then $T_\sigma$ is of finite rank.
\end{coro}

\begin{teo}
Let $\sigma \in \Phi^0(SU_q(2))$. If $\displaystyle \lim_{l \rightarrow \infty}\|\sigma(l)\|_{op}=0$ then $T_\sigma$ is a compact operator.
\end{teo}

\begin{proof}
Consider a symbol $\sigma \in \Phi^0(SU_q(2))$ and the sequence of functions $(g_n)_{n\in \mathbb{N}}$, where  $g_n: \frac{1}{2}\mathbb{N}\rightarrow \{0,1\}$  is defined by $g_n(l)=0$ for $l>n$, and $g_n(l)=1$ for $0\leq l\leq n$. Then $T_{g_n \sigma}$ has finite rank and, using Plancherel identity and the Hilbert-Schmidt norm inequality, we have that $$\|(T_\sigma-T_{g_n \sigma})f \|_{SU_q(2)}^2=h((T_\sigma-T_{g_n \sigma})f ((T_\sigma-T_{g_n \sigma})f)^{*})$$ $$=\sum_{l>n}[2l+1]_qTr( \sigma(l)^{*}\sigma(l) \hat{f}(l)(\hat{f}(l))^{*})  $$ $$\leq\sum_{l>n}[2l+1]_q\|\sigma(l) \|_{op}Tr(  \hat{f}(l)(\hat{f}(l))^{*}) $$ 
$$\leq \sup_{l>n}\|\sigma(l) \|_{op}\sum_{l>n}[2l+1]_qTr(  \hat{f}(l)(\hat{f}(l))^{*}) $$ 
$$\leq \sup_{l>n}\|\sigma(l) \|_{op}\|f\|_{SU_q(2)}^2 .$$ From this we see that  $\|(T_\sigma-T_{g_n \sigma})\| _{op} \rightarrow 0$ as $n\rightarrow \infty,$ thus  $T_\sigma$ is a compact operator
\end{proof}

In some special cases it is possible to find, from information on the symbol, the eigenvalues of the corresponding operator $T_\sigma.$
\begin{teo}
Let $\lambda: \frac{1}{2}\mathbb{N}\rightarrow \mathbb{C}$. Suppose that for each  $-l\leq i_0\leq l$  the symbol $\sigma \in \Phi^0(SU_q(2))$ satisfies that $\sum_j\sigma_{i_0j}(l)=\lambda(l)$.Then  $(\lambda(l))_{l\in \frac{1}{2}\mathbb{N}\ }\subseteq spec(T_\sigma).$ Furthermore, the multiplicity of $\lambda(l)$ is greater or equal than $2l+1$.

\end{teo}

\begin{proof}
 Suppose that $\sigma $ satisfies the condition of the theorem. Then $$T_\sigma(\sum_j t_{ij}^l)=\sum_j\left(\sum_{0\leq r\leq 2l} \sigma_{l-r,j}(l)t_{i,r-l}^l\right)$$ $$=\sum_{0\leq r\leq 2l} \left(\sum_{j} \sigma_{l-r,j}(l)\right)t_{i,r-l}^l=\lambda(l)\sum_j t_{ij}^l.$$
\end{proof}
Notice that this last result holds in the case of a diagonal matrix symbol.
\\ \\
Recall that the {\em index} of a Fredholm operator $T: \mathcal{H} \to \mathcal{H}$ acting on a Hilbert space  $\mathcal{H}$, is defined as $${\rm ind}\, (T):= \dim(\ker T)- \dim({\rm coker}\, T).$$ Using the notion of Fourier order we can compute directly the index of particular classes of global pseudo-differential operators.

\begin{teo}
Let $\sigma \in \Phi^m(SU_q(2))$ be a  symbol of  Fourier order $m\in \frac{1}{2}\mathbb{N}$ such that  $\sigma(l)_{ij}\in \textsf{Span}\{t^m_{rs}: -m\leq r,s\leq  m\}$ for $0\leq  l \leq N-\frac{1}{2}$,  and $\sigma(l)_{ij}\in \textsf{Span}\{1\}$ for $l\geq  N$ for some half-natural number N. Then  $T_\sigma$ is a Fredholm operator and its index is given by $${\rm ind}\,(T_\sigma) ={4 \over 3}m^2(m-1) + 4N(2N+Nm-1).$$   
\end{teo}
  
\begin{proof}

Let $H^l:=\textsf{Span}\{t_{ij}^l: -l\leq i,j\leq l\}$ and let $\sigma$ be a homogeneous symbol satisfying the hypothesis of the theorem. Observe first that,  for  any $\beta \in \Phi^0(SU_q(2)) $,  we have   ${\rm ind}\, T_\beta=0$. This is just because we can think that the pseudo-differential operator is direct sum of linear operator acting on the finite dimensional spaces $H^l.$ Then, by the Clebsch-Gordan decomposition of the products 
 we can see $$T_\sigma \big(\displaystyle\bigoplus_{0\leq N} H^l\big)\subseteq \displaystyle \bigoplus_{0\leq l\leq N+m} H^l$$ and also that $T_\sigma(H^l)\subseteq H^l,$  from which we conclude that $${\rm ind}\, (T_\sigma) \rvert _{\displaystyle\bigoplus_{0\leq N} H^l} + \,\,  \sum_{l\geq N } {\rm ind}\,(T_\sigma) \lvert _{H^l}=\sum_{l=N}^{N+m}(2l+1)^2,$$
and the explicit computation of the sum gives the result.
\end{proof}


\begin{rem}
We finally point out that, in the definition of global pseudo-differential operators, we used the natural extension of the definition of the trace of a matrix. However, this is {\em not } a trace if we consider matrices with entries in  the non-commutative algebra $SU_q(2)$. In order to know what are the traces for these kind of matrices,  we recall  that the linear operator $Tr^0: SU_q(2)\rightarrow \mathbb{C}$ defined  by $Tr^0(f)=\hat{f}(0)$ is a trace (called {\em the non-commutative integral}), and we define for each $l\in \frac{1}{2}\mathbb{N}$ the complex valued  operator $Tr^l:M_{2l+1\times 2l+1}(\mathbb{C})\otimes SU_q(2)\rightarrow \mathbb{C}$ by $Tr^l(A)=\sum_i Tr^0(A_{ii})$. It is straightforward to see that these are traces and that in fact they are the unique traces satisfying $Tr(1)=1$, and $Tr^l(I_{2l+1\times 2l+1})=2l+1$ for all $l\in \mathbb{N}$. Both traces and determinants for global pseudo-differential operators on quantum groups will be considered in a separate paper.

\end{rem}


\end{document}